\documentclass{amsart}

\newtheorem{theorem}{Theorem}[section]

\newtheorem{prop}[theorem]{Proposition}
\newtheorem{cor}[theorem]{Corollary}
\newtheorem*{thm}{Theorem}

\theoremstyle{definition}
\newtheorem{definition}[theorem]{Definition}
\newtheorem*{ack}{Acknowledgement}

\usepackage{enumerate}

\theoremstyle{remark}

\numberwithin{equation}{section}
\usepackage{mathtools}
\usepackage{mathrsfs}
\usepackage{float} 
\usepackage{graphicx}
\usepackage{tikz-cd}
\usepackage{float}
\usepackage{hyperref}
\usetikzlibrary{calc, intersections}
\begin{document}
	
	\title{General position on Brauer--Severi surfaces}
	
	\author{Jack Ritschel}
	\address{Heinrich Heine University Dusseldorf, Faculty of Mathematics and Natural Sciences, Mathematical Institute, 40204 Dusseldorf, Germany }
	\email{jack.ritschel@hhu.de}
	
	\begin{abstract}
		The blowing-up of the projective plane at a finite set of points yields a del Pezzo surface if and only if the points lie in general position. In this note, we generalize this result to Brauer--Severi surfaces over arbitrary fields. Using Galois descent, intersection theory and combinatorial arguments, we provide explicit arithmetic and geometric conditions on the centre of the blowing-up. 
	\end{abstract}
	
	\maketitle
	\tableofcontents
	
	\section*{Introduction}
	
	Smooth algebraic curves, surfaces, threefolds, etc., over a field $k$ can be broadly grouped into three classes, depending on whether their dualizing sheaf is ample, trivial, or anti-ample. An algebraic surface $X$ is called a \textit{del Pezzo} surface if it is smooth, proper, and its dualizing sheaf $\omega_X$ is anti-ample. The most basic example is the projective plane, with $\omega_{\mathbb{P}^2} = \mathscr{O}_{\mathbb{P}^2}(-3)$. 
	
	It is a well-understood problem to determine the conditions under which the blowing-up $ \operatorname{Bl}_{Z}(\mathbb{P}^2_k) \rightarrow \mathbb{P}^2_k$ at $r \ge 1$ rational points $Z = \{a_1, \ldots, a_r\}$ yields a del Pezzo surface. It turns out the points must lie in a certain configuration, called  \textit{general position}. We recall this in Theorem~\ref{P^2_del_Pezzo}.
	
	There exist twisted forms of projective space, the so-called \textit{Brauer--Severi varieties}. These are $k$-schemes for which there exists a field extension $k \subset L$ such that the base change becomes isomorphic to $\mathbb{P}^n_L$. When the Brauer group $\operatorname{Br}(k)$ of the field is not trivial, there may exist many such non-isomorphic schemes over $k$. In dimension two, we speak of Brauer--Severi surfaces. 
	
	Since the ampleness of invertible sheaves can be checked after base change, Brauer--Severi surfaces are always del Pezzo surfaces. It is therefore natural to ask how the classical criteria adapt to this setting. In this note, we determine the conditions under which the blowing-up of a Brauer--Severi surface $\operatorname{Bl}_Z(X) \rightarrow X$ at a collection $Z = \{a_1, \ldots, a_r\}$  of $r \ge 1$ closed points yields a del Pezzo surface. 
	
	Here we consider closed points endowed with their reduced induced subscheme structure. We exclude non-reduced schemes as centres for blowing-ups, as this often results in singular surfaces.
	
	Brauer--Severi surfaces have different Chow groups $\operatorname{CH}(X)$ than the projective plane. In particular, the degree of any closed point $a \in X$ is a multiple of 3 (see Proposition~\ref{Non-Split-Chow}). Consequently, one is forced to blow up at non-rational points, which is a delicate matter. Because blowing-up commutes with flat base change, we can pass to the separable closure $k^{\mathrm{sep}}$, over which the surface $X_{k^{\mathrm{sep}}}\simeq\mathbb{P}^2_{k^{\mathrm{sep}}}$ splits. Thus we determine conditions under which the points $u^{-1}(Z)$ in the preimage of $u: X_{k^{\mathrm{sep}}} \rightarrow X$ lie in general position. 
	
	We introduce the notion of \textit{geometric general position} to refer to such $a_i \in Z$. This is related to, but different from \textit{Sarkisov general position} in \cite{Bernasconi et al 2024}. We give a short comparison after Proposition~\ref{Smooth_Blowing_Up}.
	
	In the first section, we work out restrictions on $Z = \{a_1, \ldots, a_r\}$. By Proposition~\ref{Smooth_Blowing_Up}, the residue fields $k \subset \kappa(a_i)$ must all be separable. It follows from the classical version that the total degree of the points is bounded by 8 (see Proposition~\ref{degree_bounded}). Since every point also has degree divisible by 3, this leaves only three possibilities: $Z = \{a\}$ may consist of a single point of degree 3 or 6, or $Z = \{a, a^{\prime}\}$ consists of two points, of degree 3 each.
	
	In the second section, we turn to these three cases. In \cite{Weinstein 2022}, Weinstein showed that blowing up at a single point $Z = \{a\}$ of degree 3 or 6 always yields a del Pezzo surface. What is new is the case $Z = \{a, a^{\prime}\}$. Here, the situation turns out to be more complicated. The conditions now depend on the arithmetic of the residue fields, and the configuration of the points in the preimage is more involved. In particular, a certain nodal curve plays a distinguished role. It arises as follows: The preimage of any degree-3 point $a \in X$ is a non-collinear triple. This triple yields a triangle $C_a$ of lines through the points. By descent, we get a nodal curve $C_a \subset X$, whose properties we examine in Proposition~\refeq{nodal_curve}. 
	
	Summarising, when the centre $Z = \{a_1, \ldots, a_r\} \subset X$ is reduced and consists of closed points, we arrive at the following conditions: 
	
	\begin{thm}[see Thm.~\refeq{main_theorem}]
		Let $X$ be a non-split Brauer--Severi surface over $k$. The blowing-up $\operatorname{Bl}_Z(X) \rightarrow X$ along the centre $Z$ is a del Pezzo surface if and only if the residue field of each point is separable and one of the following holds: 
		\begin{enumerate}[(i)]
			\item $Z = \{a\}$ consists of a degree-3 or 6 point. 
			\item $Z = \{a, a^{\prime}\}$ consists of two degree-3 points, with $\kappa(a) \not\simeq \kappa(a^{\prime})$. 
			\item $Z = \{a, a^{\prime}\}$ consists of two degree-3 points, with $\kappa(a) \simeq \kappa(a^{\prime})$ such that $a^{\prime} \notin C_a$ and $a \notin C_{a^\prime}$.
		\end{enumerate}
	\end{thm}

	\begin{ack}I would like to express my gratitude to S.~Schroeer for suggesting this topic and for many helpful discussions throughout the preparation of this note, which was a part of my master's thesis. I am also grateful to O.~Overkamp for taking the time to discuss mathematical questions with me. Finally, I would like to thank the anonymous referee for their helpful suggestions, in particular for suggesting the argument in the proof of Proposition~\ref{6pts_general_position}.
	\end{ack}

	\section{From the projective plane to Brauer--Severi surfaces}
	
	To fix notation and to set up the classical result, we quickly collect some basic definitions and results in the theory of algebraic surfaces. A good reference for this section is \cite[Ch. 9 $\S$2]{Liu 2002}. Fix a field $k$ of characteristic $p \ge 0$. By an algebraic surface $X$, we mean a separated, finite type $k$-scheme of dimension two. 
	
	For proper algebraic surfaces, the numerical group $\operatorname{Num}(X) = \operatorname{Pic}(X) / \operatorname{Pic}^{\tau}(X)$ of invertible sheaves modulo those which are numerically trivial comes with an intersection form, where for any two Cartier divisors $D,D^{\prime}$ we get an integer $(D\cdot D^{\prime})$.
	
	Consider the blowing-up $b: S = \operatorname{Bl}_{z}(X) \rightarrow X$ at a closed point $z \in X$, with \textit{exceptional divisor} $E = b^{-1}(z)$, isomorphic to the projective line over the residue field. Note that the self-intersection $E^2 = -[\kappa(z):k]$ keeps track of the degree. When $X$ is proper and regular, there is an isomorphism of numerical groups $\operatorname{Num}(S) = b^*\operatorname{Num}(X) \oplus \mathbb{Z}E$. This extends to blowing-up multiple points, with one exceptional divisor $E_i$ per point, $(E_i \cdot E_j) = 0$ for $i\neq j$ and $K_S = b^*K_X + \sum_i E_i$. 
	
	Let $X = \mathbb{P}^2_k$, let $Z = \{z_1, \ldots, z_r\}$ consist of closed points, with blowing-up $S = \operatorname{Bl}_Z(X)$. Identifying numerical groups, the canonical divisor can be written as $K_S = -3H + E_1 + \ldots + E_r$, with $H$ being the hyperplane divisor. By the \textit{Nakai--Moishezon} criterion \cite[p.365]{Hartshorne}, $S$ is a del Pezzo surface if and only if $K_S^2 > 0$, and there are no curves $C \subset S$ with the property that $(K_S \cdot C) \ge 0$. 
	 
	When the centre $Z = \{z_1, \dots, z_r\} \subset \mathbb{P}^2_k$ consists of $k$-rational points, the following classical theorem, found, for example, in \cite[Proposition~2]{Demazure}, gives conditions on the number and geometry of points under which the blowing-up is a del Pezzo surface: 
	
	\begin{theorem}\label{P^2_del_Pezzo}
		The blowing-up $S$ is a del Pezzo surface if and only if $r < 9$, no three or more points lie on a line, no six or more lie on a conic, and no seven lie on a singular cubic with an eighth point as a double point.
	\end{theorem}
	If these conditions are satisfied, one also speaks of the points lying in \textit{general position}.
	
	Now, let $X$ be a twisted form of $\mathbb{P}^2_k$, that is, a Brauer--Severi surface. If $X \not\simeq \mathbb{P}^2_k$, we call $X$ \textit{non-split}. A characterising property of non-split surfaces is that they admit only a restrictive class of closed subschemes. Indeed, the curve class is generated by $-K_X$. 
	The group of zero-cycles is generated by the vanishing locus of a general section of the tangent sheaf $\mathscr{T}_X$, see \cite[Theorem 36]{Kollar 2025}. Over infinite fields, this cuts out a reduced closed point of degree three. Over finite fields, this is not guaranteed. However, there is no harm; every Brauer--Severi surface is split by Wedderburn's theorem, since $\operatorname{Br}(\mathbb{F}_q) = 0$ \cite[Ch. X $\S$7]{Serre Local Fields}. Summarizing, we get:
	
	\begin{prop}\label{Non-Split-Chow}
	Let $X$ be a  Brauer--Severi surface. The following are equivalent:
	\begin{enumerate}[(i)]
		\item $X$ is non-split. 
		\item $\operatorname{CH}_1(X) = (-K_X)\mathbb{Z}$ is generated by the anti-canonical divisor.
		\item $\operatorname{CH}_0(X) = c_2(\mathscr{T}_X)\mathbb{Z}$ is generated by a single, degree-3 point.
	\end{enumerate} 
	\end{prop}
	
	As a consequence, a non-split Brauer--Severi surface does not contain any curve or closed point of degree coprime to 3.  Therefore, we are tasked with blowing up non-rational points, which is a delicate matter. 
	
	We can, however, reduce to the classical case. Indeed, since blowing-up commutes with flat base change (\cite[0805]{Stacks}) and ampleness can be checked after field extension (\cite[0D2P]{Stacks}), we can apply Theorem~\ref{P^2_del_Pezzo} to the case of Brauer--Severi surfaces $X$ after passing to the separable closure $u:X_{k^{\mathrm{sep}}} \rightarrow X$. Concretely, let $Z = \{a_1, \dots, a_r\}$ consist of $r \ge 1$ closed points. As $X_{k^{\mathrm{sep}}}$ is isomorphic to $\mathbb{P}^2_{k^{\mathrm{sep}}} $, it suffices to check whether
	\begin{equation*}
		\operatorname{Bl}_Z(X)_{k^\mathrm{sep}} = \operatorname{Bl}_{u^{-1}(Z)} (X_{k^{\mathrm{sep}}}) 
	\end{equation*} is a del Pezzo surface. Therefore, we now seek to understand when the points in the preimage $u^{-1}(Z) \subset X_{k^{\mathrm{sep}}}$ lie in general position, i.e. if they fulfil the conditions of Theorem~\ref{P^2_del_Pezzo}. If they do, we say that the points $a_i \in Z$ lie in \textit{geometric general position}. 
	
	An immediate constraint on the residue fields $k \subset \kappa(a_i)$ is that they must all be separable. Of course, this is only relevant in characteristic $p > 0$. We prove something more general:
	
	\begin{prop}\label{Smooth_Blowing_Up}
		Let $X$ be a smooth scheme of finite type over $k$ of pure dimension $n \ge 2$. Let $a \in X$ be a closed point whose residue field $k \subset \kappa(a)$ is not a separable extension. Then the blowing-up $\operatorname{Bl}_a(X)$ is not smooth. 
	\end{prop}
	
	\begin{proof}
		Recall that a finite type $k$-scheme $X$ is smooth if and only if it is geometrically regular at all points (\cite[038X]{Stacks}). We aim to show that $\operatorname{Bl}_a(X)$ is not geometrically regular. 
		
		Let $L = \kappa(a)$. First, we reduce to purely inseparable field extensions as follows. Any finite extension $k \subset L$ factors into a separable and purely inseparable part \[ [L:k] = [L_{\mathrm{sep}}:k] \cdot [L:L_{\mathrm{sep}}] = m \cdot p^n, \] where $L_{\mathrm{sep}}$ is the separable closure of $k$ in $L$, and $m,n \ge 1$ (\cite[030K]{Stacks}). For the base change $u: X_L \rightarrow X$, the preimage $u^{-1}(a) = \{z_1, \ldots, z_m\}$ consists of $m$ points of length $p^n$ each. It suffices to show that blowing up at one of these $z_i$ results in a singular surface. We may thus assume that $k = k^{\mathrm{sep}}$ and $[L:k] = p^n$ is purely inseparable. This way, $u^{-1}(a) = \{z\} \subset X_L$ consists of a single point. 
		
		The question reduces to a problem in local algebra. Let $X = \operatorname{Spec}(A)$ be the spectrum of the regular local ring $A = \mathscr{O}_{X,a}$. Denote its maximal ideal by $\frak{m}$. The base change is $X_L = \operatorname{Spec}(A \otimes_k L)$, and the schematic preimage $u^{-1}(a)$ is the spectrum of the quotient $A \otimes_k L / \frak{m} \otimes_k L$.
		
		We apply \cite[Theorem~2.1]{O'Carrol et al 1997}, whose proof extends verbatim to characteristic $p > 0$. The assertion states that the blowing-up $\operatorname{Bl}_{u^{-1}(a)}(X_L) = \operatorname{Bl}_a(X) \otimes_k L$ is smooth if and only if the quotient $A \otimes_k L / \frak{m} \otimes_k L = L \otimes_k L$ is regular. As $k \subset L$ is purely inseparable, $L \otimes_k L$ contains non-zero nilpotents. Consequently, the quotient is not regular. We conclude that $\operatorname{Bl}_a(X)$ is not geometrically regular, hence not smooth over $k$.
	\end{proof}

	In \cite[$\S$4.3]{Bernasconi et al 2024} the authors introduce the notion of \textit{Sarkisov general position}. Since they require del Pezzo surfaces to be only regular, being in geometric general position is a stronger property over imperfect fields, by the result above.
	
	We return to Brauer--Severi surfaces $X$ and $Z = \{a_1, \ldots, a_r\} \subset X$ consisting of closed points. In the split case, with all points $k$-rational, the total number of points must be less than 9, by Theorem~\ref{P^2_del_Pezzo}. While non-split Brauer--Severi surfaces do not contain $k$-rational points, there is an analogue: 
	
	\begin{prop}\label{degree_bounded}
		Let $X$ be a Brauer--Severi surface with centre $Z$ as above. If the total degree $\deg(a_1) + \ldots + \deg(a_r) \ge 9$, the blowing-up $S = \operatorname{Bl}_{Z}(X)$ cannot be a del Pezzo surface. 
	\end{prop}
	
	\begin{proof}
		Similarly to the projective plane, we can identify the canonical divisor as
		$K_S = b^*K_X + E_1 + \ldots + E_r$, with $E_i^2 = -[\kappa(a_i):k]$. For $-K_S$ to be ample, we must ensure that $K_S^2 = 9 - \sum_i [\kappa(a_i):k] > 0$. The statement follows.
	\end{proof}
	
	Having collected all intermediate results, we now determine the possible blowing-up centres and examine them in detail. 
	
	\section{Points in geometric general position}
	
	For a non-split Brauer--Severi surface $X$ and $Z = \{a_1, \ldots, a_r\}$ a finite set of closed points, we wish to understand whether $\operatorname{Bl}_Z(X)$ is a del Pezzo surface. Equivalently, we ask whether the blowing-up of the pullback $\operatorname{Bl}_{u^{-1}(Z)}(X_{k^{\mathrm{sep}}})$ along the preimage $u^{-1}(Z)$ is a del Pezzo surface. That is, we ask when the points $u^{-1}(Z) \subset X_{k^{\mathrm{sep}}}$ lie in general position in the sense of Theorem~\ref{P^2_del_Pezzo}. 
	
	The possible blowing-up centres $Z$ are highly restricted. Indeed, by Proposition~\ref{Non-Split-Chow} and Proposition~\ref{degree_bounded}, respectively, we have $3 \mid \deg(a_i)$, and the total degree is at most 8. Thus, we are left with the following three possibilities: $Z = \{a\}$ may consist of a single point of degree 3 or 6, or $Z = \{a, a^{\prime}\}$ consists of two points of degree 3 each. Further, by Proposition~\ref{Smooth_Blowing_Up}, the respective residue fields are always separable. 
	
	Consider a degree-3 point $a \in X$. Suppose the points in the preimage $u^{-1}(a) = \{z_1, z_2, z_3\} \subset X_{k^{\mathrm{sep}}}$ are contained in a line $L \subset X_{k^{\mathrm{sep}}}$. The triple inherits a transitive action by the absolute Galois group $G = \operatorname{Gal}(k^{\mathrm{sep}}/k)$. Then, the line is stabilised as well, and by descent, the degree-1 curve $L \subset X$ is defined over $k$ -- a contradiction, by Proposition~\ref{Non-Split-Chow}. The remaining two cases require some extra care, since six points can fail to be in general position in numerous ways. To tackle this systematically, we call a line $L \subset X_{k^{\mathrm{sep}}}$ \textit{rich}, if it contains at least three of the six points in the preimage. We claim that there are exactly nine configurations of six points containing at least one rich line: 
	
		\begin{figure}[H]
			\centering
			\begin{tikzpicture}[scale=.55]
				\tikzset{
					dot/.style={circle, fill=black, inner sep=0pt, minimum size=4pt},
					base line/.style={draw=black, thick, shorten <=-3mm, shorten >=-3mm}
				}
				
				
				\begin{scope}[shift={(0,0)}]
					\node[anchor=north west] at (-0.5, 2.5) {\textbf{A}};
					\draw[base line] (0, 0) -- (3.0, 0); 
					\foreach \x in {0, 0.6, 1.2, 1.8, 2.4, 3.0} { 
						\node[dot] at (\x, 0) {}; 
					}
				\end{scope}
				
				\begin{scope}[shift={(5,0)}]
					\node[anchor=north west] at (-0.5, 2.5) {\textbf{B}};
					\draw[base line] (0, 0) -- (2.4, 0);
					\foreach \x in {0, 0.6, 1.2, 1.8, 2.4} { 
						\node[dot] at (\x, 0) {}; 
					}
					\node[dot] at (1.2, 1) {}; 
				\end{scope}
				
				\begin{scope}[shift={(10,0)}]
					\node[anchor=north west] at (-0.5, 2.5) {\textbf{C}};
					\draw[base line] (0, 0) -- (1.8, 0);
					\foreach \x in {0, 0.6, 1.2, 1.8} { \node[dot] at (\x, 0) {}; }
					\node[dot] at (0.3, 1) {};
					\node[dot] at (1.5, 1) {};
				\end{scope}
				
				
				\begin{scope}[shift={(0,-4)}]
					\node[anchor=north west] at (-0.5, 2.5) {\textbf{D}};
					\draw[base line] (0, 0) -- (1.2, 0);
					\foreach \x in {0, 0.6, 1.2} { \node[dot] at (\x, 0) {}; }
					\node[dot] at (0.0, 1.0) {}; 
					\node[dot] at (0.6, 1.5) {}; 
					\node[dot] at (1.2, 1.2) {}; 
				\end{scope}
				
				\begin{scope}[shift={(5,-4)}]
					\node[anchor=north west] at (-0.5, 2.5) {\textbf{E}};
					\draw[base line] (0, 0) -- (1.2, 0);
					\foreach \x in {0, 0.6, 1.2} { \node[dot] at (\x, 0) {}; }
					\draw[base line] (0, 1.5) -- (1.2, 1.5);
					\foreach \x in {0, 0.6, 1.2} { \node[dot] at (\x, 1.5) {}; }
				\end{scope}
				
				\begin{scope}[shift={(10,-4)}]
					\node[anchor=north west] at (-0.5, 2.5) {\textbf{F}};
					\draw[base line] (0, 0) -- (1.8, 0);
					\foreach \x in {0, 0.6, 1.2, 1.8} { \node[dot] at (\x, 0) {}; }
					\draw[base line] (0, 0) -- (45:1.2);
					\foreach \r in {0.6, 1.2} { \node[dot] at (45:\r) {}; }
				\end{scope}
				
				
				\begin{scope}[shift={(0,-8)}]
					\node[anchor=north west] at (-0.5, 2.5) {\textbf{G}};
					\draw[base line] (0, 0) -- (1.2, 0);
					\foreach \x in {0, 0.6, 1.2} { \node[dot] at (\x, 0) {}; }
					\draw[base line] (0, 0) -- (45:1.2);
					\foreach \r in {0.6, 1.2} { \node[dot] at (45:\r) {}; }
					\node[dot] at (1.5, 0.8) {};
				\end{scope}
				
				\begin{scope}[shift={(5,-8)}, scale=0.75]
					\node[anchor=north west] at (-0.5, 3.3) {\textbf{H}}; 
					
					\coordinate (h1) at (0, 0);
					\coordinate (h2) at (2.5, 0);    
					\coordinate (h3) at (1.25, 2.1); 
					
					\draw[base line] (h1) -- (h2);
					\draw[base line] (h2) -- (h3);
					\draw[base line] (h3) -- (h1);
					
					\node[dot] at (h1) {}; \node[dot] at (h2) {}; \node[dot] at (h3) {};
					
					\node[dot] at ($(h1)!0.3!(h2)$) {};
					\node[dot] at ($(h2)!0.7!(h3)$) {};
					\node[dot] at ($(h3)!0.4!(h1)$) {};
				\end{scope}
				
				\begin{scope}[shift={(10,-8)}, scale=0.75]
					\node[anchor=north west] at (-0.5, 3.3) {\textbf{I}}; 
					
					\coordinate (i1) at (0,0);   
					\coordinate (i2) at (1.2,0); 
					\coordinate (i3) at (2.4,0); 
					
					
					\coordinate (p23) at (intersection of i1--{$(i1)+(55:5)$} and i2--{$(i2)+(80:5)$});
					\coordinate (p34) at (intersection of i2--{$(i2)+(80:5)$} and i3--{$(i3)+(145:5)$});
					\coordinate (p24) at (intersection of i1--{$(i1)+(55:5)$} and i3--{$(i3)+(145:5)$});
					
					\draw[base line] (i1) -- (i3);       
					\draw[base line] (i1) -- (p23);      
					\draw[base line] (i2) -- (p23);      
					\draw[base line] (i3) -- (p24);      
					
					\node[dot] at (i1) {}; \node[dot] at (i2) {}; \node[dot] at (i3) {};
					\node[dot] at (p23) {};
					\node[dot] at (p34) {};
					\node[dot] at (p24) {};
				\end{scope}
				
				\path (0,-9);
			\end{tikzpicture}
			
			\caption{Six points with at least one rich line.}
		\end{figure}

	To see that this exhausts all possibilities, we use a combinatorial argument, ordered by the number $d \ge 1$ of rich lines. 
	\begin{enumerate}[(1)]
		\item There is exactly one rich line in configurations \textsc{A}, \textsc{B}, \textsc{C} and \textsc{D}.
		\item  The configurations containing exactly two rich lines are \textsc{E}, \textsc{F} and \textsc{G}.
		\item The only configuration containing exactly three rich lines is \textsc{H}.
		\item Similarly, \textsc{I} is the only configuration containing four rich lines. Since two rich lines through a point already account for four of the remaining five points, each point lies on at most two rich lines, so counting incidences gives at most four rich lines in total. 
	\end{enumerate}

	The following Proposition extends a result of Weinstein \cite{Weinstein 2022}. Perhaps surprisingly, the arithmetic of the residue fields now comes into play: 
	
	\begin{prop}\label{6pts_general_position}
		Let $X$ be a non-split Brauer--Severi surface. Let $Z \subset X$ be a centre whose points have separable residue fields. Suppose that either
		\begin{enumerate}[(i)]
			\item  $Z = \{a\}$ consists of a single point of degree 3 or 6, or 
			\item $Z = \{a, a^{\prime}\}$ consists of two points of degree 3 each, with $\kappa(a) \not\simeq \kappa(a^{\prime})$.
		\end{enumerate}
	Then the points of $Z$ lie in geometric general position.
	\end{prop}
	
	\begin{proof}
		The condition on singular cubics, which concerns eight points, is ruled out immediately, as we are only blowing up at a maximum of six points. All remaining cases rest on the following observations: The Galois action permutes rich lines. Therefore, the number of rich lines through each point is constant on each Galois orbit. Further, the union of the rich lines descends to $X$ as a curve whose degree $d \ge 1$ is equal to the number of rich lines. Then, if $d$ is coprime to 3, $X$ is split via Proposition~\ref{Non-Split-Chow}. 
		
		An irreducible conic containing the six points is unique, since it meets any other irreducible conic in at most four points. Thus, it is stabilised by the Galois action, and descends to a degree-2 curve on $X$. Next, suppose there is at least one rich line. This includes the case of a reducible conic. Notice that from the above observations alone, all configurations but \textsc{H} are ruled out, since in all other cases the number of rich lines is 1, 2 or 4.
		
		\begin{enumerate}[(i)]
		\item For a degree-6 point, there is one orbit in the preimage. However, in \textsc{H} the number of rich lines through a point is not constant, so \textsc{H} is excluded. The degree-3 case was discussed previously. Consequently, the single point $a$ is always in geometric general position. 
		
		\item For two degree-3 points, there are now two orbits. Denote them by $u^{-1}(a) = \{z_1, z_2, z_3\}$ and $u^{-1}(a^{\prime}) = \{z_1^{\prime}, z_2^{\prime}, z_3^{\prime} \}$, respectively. Notice that now configuration \textsc{H} is possible:
		
	\begin{figure}[H]
		\centering
		\resizebox{.3\textwidth}{!}{
			\begin{tikzpicture}[
				>=stealth,
				thick,
				dot/.style={circle, fill=black, inner sep=0pt, minimum size=4pt},
				red dot/.style={circle, fill=red, inner sep=0pt, minimum size=3pt},
				base line/.style={thick, shorten <=-0.5cm, shorten >=-0.5cm},
				red dashed line/.style={thick, red, dashed, shorten <=-0.2cm, shorten >=-0.2cm}
				]
				
				\begin{scope}[shift={(5,-8)}, scale=0.75]
					
					\coordinate (z3) at (0, 0);     
					\coordinate (z1) at (2.5, 0);   
					\coordinate (z2) at (1.25, 2.1); 
					
					\coordinate (z3_p) at ($(z3)!0.3!(z1)$); 
					\coordinate (z1_p) at ($(z1)!0.7!(z2)$); 
					\coordinate (z2_p) at ($(z2)!0.4!(z3)$); 
					
					\draw[base line] (z3) -- (z1);
					\draw[base line] (z1) -- (z2);
					\draw[base line] (z2) -- (z3);
					
					\draw[red dashed line] (z1_p) -- (z2_p);
					\draw[red dashed line] (z2_p) -- (z3_p);
					\draw[red dashed line] (z3_p) -- (z1_p);
					
					\node[dot] at (z3) {}; \node[below left=4pt] at (z3) {$z_3$};
					\node[dot] at (z1) {}; \node[below right=4pt] at (z1) {$z_1$};
					\node[dot] at (z2) {}; \node[above=4pt] at (z2) {$z_2$};
					
					\node[red dot] at (z3_p) {}; \node[below=7pt, red] at (z3_p) {$z_3'$};
					\node[red dot] at (z1_p) {}; \node[right=7pt, red] at (z1_p) {$z_1'$};
					\node[red dot] at (z2_p) {}; \node[left=7pt, red]  at (z2_p) {$z_2'$};
					
				\end{scope}
				
			\end{tikzpicture}
		}
		\caption{Configuration  \textsc{H} with both triangles drawn.}
		\label{fig:triangle-arbitrary-points}
		
	\end{figure}

		Each vertex lies on two rich lines and each side point on one. With two orbits of three points each, one orbit is the vertex set and the other the side points. Crucially, sending each vertex to the point on the opposite side gives a $G$-equivariant bijection between $u^{-1}(a)$ and $u^{-1}(a^{\prime})$, since the Galois action preserves incidence. 
		
		In other words, the preimages are isomorphic as finite, continuous $G$-sets. Thus, in the language of Grothendieck's algebraic fundamental groups (see \cite[0BND]{Stacks}), the spectra of $\kappa(a)$ and $\kappa(a^{\prime})$ are isomorphic as étale covers of $\operatorname{Spec}(k)$. Put differently, $\kappa(a) \simeq \kappa(a^{\prime})$ as étale $k$-algebras.
		
		This implies that, if $a$ and $a^{\prime}$ have non-isomorphic residue fields, configuration \textsc{H} cannot occur and, consequently, $Z = \{a, a^{\prime}\}$ is in geometric general position.
	
	\end{enumerate}
	\end{proof}
	
	Let us record this as a sufficient condition for the blowing-up of $X$ to be a del Pezzo surface:
	
	\begin{cor}
		Let $S = \operatorname{Bl}_{Z}(X) \rightarrow X$ be the blowing-up of a non-split Brauer--Severi surface at the centre $Z$ as in Proposition~\ref{6pts_general_position}. Then $S$ is a del Pezzo surface.
	\end{cor}
	
	Note that this is not necessarily true in the split case. Indeed, if $a \in \mathbb{P}^2_k$ is a degree-3 point, it may very well be contained in a line $L \simeq \mathbb{P}^1_k$ defined over $k$, in contrast to non-split surfaces. When this happens, the points in the preimage become collinear. 
	
	In order to understand the case $\kappa(a) \simeq \kappa(a^{\prime})$ better, we need to look at the respective triangle of lines. Each such triangle descends to a curve in $X$. One can think of these curves as being attached to the points, or their residue fields, as follows:
	
	\begin{definition}
		Let $X$ be a non-split Brauer--Severi surface. Let $a \in X$ be a closed point of degree 3, with preimage $u^{-1}(a) = \{z_1, z_2, z_3\}$ and lines $L_{12}, L_{23}, L_{13} \subset X_{k^{\mathrm{sep}}}$ through each pair of points, stable under the action of the Galois group. The curve $C_a \subset X$ defined by descending the triangle is the \textit{nodal curve associated with} $\kappa(a)$. 
	\end{definition}

	We write $C \subset X$ when it is clear which $a \in X$ is meant. The curve may be visualised as follows:
	
	\begin{figure}[H]
		\centering
		\resizebox{.6\textwidth}{!}{
			\begin{tikzpicture}[
				>=stealth,
				thick,
				dot/.style={circle, fill=black, inner sep=0pt, minimum size=4pt},
				fat dot/.style={circle, fill=black, inner sep=0pt, minimum size=8pt},
				base line/.style={thick, shorten <=-0.5cm, shorten >=-0.5cm},
				strand/.style={line width=0.6pt, line cap=butt}
				]

				\begin{scope}[scale=0.75]
					\coordinate (z3) at (0, 0);
					\coordinate (z1) at (2.5, 0);
					\coordinate (z2) at (1.25, 2.1);

					\draw[base line] (z3) -- (z1);
					\draw[base line] (z1) -- (z2);
					\draw[base line] (z2) -- (z3);

					\node[dot] at (z3) {}; \node[below left=4pt] at (z3) {$z_3$};
					\node[dot] at (z1) {}; \node[below right=4pt] at (z1) {$z_1$};
					\node[dot] at (z2) {}; \node[above=4pt] at (z2) {$z_2$};

				\end{scope}

				\draw[->] (2.6, 0.8) -- node[above] {$u$} (3.8, 0.8);

				\begin{scope}[shift={(5.2, 0.8)}]
					\foreach \d in {-0.06, 0, 0.06} {
						\draw[strand] plot[domain=-1.45:1.45, samples=120, smooth]
							({1.1*(\x*\x - 1) - \d*0.55*(3*\x*\x - 1)/sqrt((2.2*\x)^2 + (0.55*(3*\x*\x - 1))^2)},
							 {0.55*(\x*\x*\x - \x) + \d*2.2*\x/sqrt((2.2*\x)^2 + (0.55*(3*\x*\x - 1))^2)});
					}
					\node[fat dot] (a) at (0, 0) {};
					\node[right=6pt] at (a) {$a$};
				\end{scope}

				\node[anchor=base] at (0.94, -0.9) {$C_{k^{\mathrm{sep}}}$};
				\node[anchor=base] at (5.25, -0.9) {$C$};
			\end{tikzpicture}
		}
		\caption{The triangle $C_{k^{\mathrm{sep}}}$ and the nodal curve $C \subset X$ associated with $\kappa(a)$.}
		\label{fig:nodal-curve}
	\end{figure}

	The name is justified by the following proposition, which also records some more properties of the curve:
	
	\begin{prop}\label{nodal_curve}
		Let $a \in X$ and $C \subset X$ as above. 
		
		\begin{enumerate}[(i)]
			\item The curve $C$ is an integral, geometrically reduced curve of genus 1, whose geometric irreducible components are lines. Further, it is nodal with unique singularity $a$.  
			\item Every closed point $a^{\prime} \in C$ of degree 3 has residue field $\kappa(a^{\prime}) \simeq \kappa(a)$.
		\end{enumerate}	
	\end{prop}

	\begin{proof}
		\begin{enumerate}[(i)]

			\item Geometrically, $C_{k^{\mathrm{sep}}} = L_{12} \cup L_{23} \cup L_{13}$ is a triangle, whose pairwise intersections are exactly $u^{-1}(a) = \{z_1, z_2, z_3\}$. Since the Galois action on the three lines is transitive, the descended curve is integral. The genus can be checked after base change.
			
			The complement $C \setminus \{a\}$ is geometrically regular, since away from the vertices, $C_{k^{\mathrm{sep}}}$ is locally a line. So $C$ is regular away from the unique singularity $a$. By definition (\cite[0C47]{Stacks}), $a \in C$ is a node if there is a geometric ordinary double point mapping to it. To see this, notice that the points in the preimage $z_1, z_2, z_3 \in C_{k^{\mathrm{sep}}}$ lie on two irreducible lines. That is, each point $z_i$ in the preimage has a local description $\widehat{\mathscr{O}}_{C_{k^{\mathrm{sep}}}, z_i} = k^{\mathrm{sep}}[[x,y]]/(xy)$. 
			
			\item Consider the normalisation $\hat{C} \rightarrow C$. Let $\hat{k} = H^0(\hat{C}, \mathscr{O}_{\hat{C}})$. Geometrically, $\hat{C}_{k^{\mathrm{sep}}}$ is a disjoint union of three lines, with transitive Galois action. Equivalently, $C$ is irreducible over $k$. It follows that $k \subset \hat{k}$ is a separable, degree-3 extension. Since the preimage $u^{-1}(a) = \{z_1, z_2, z_3\}$ of the node is stabilised by the same Galois action, it follows that $\kappa(a) \simeq \hat{k}$. 
			
			Any other degree 3 point $a^{\prime} \in C$ lies in the smooth locus. Since the normalization map is a local isomorphism, we get a unique closed point $\hat{a}' \in \hat{C}$ mapping to $a^{\prime}$, with isomorphic residue fields. We get an inclusion $k \hookrightarrow \hat{k} \hookrightarrow \kappa(a^{\prime})$. As $[\kappa(a^{\prime}): k] = 3$, and $[\hat{k}:k] = 3$, this yields $[\kappa(a^{\prime}):\hat{k}] = 1$. We conclude $\kappa(a^{\prime}) \simeq \kappa(a)$. \qedhere
		\end{enumerate}
	\end{proof}

	Notice that Proposition~\ref{nodal_curve}(ii) gives another proof of Proposition~\ref{6pts_general_position}(ii). We are ready to state and prove the main theorem. Putting it all together, when $Z \subset X$ is a finite, reduced closed subscheme consisting of closed points, we obtain the following
	
	\begin{theorem}\label{main_theorem}
		Let $X$ be a non-split Brauer--Severi surface over $k$. The blowing-up $\operatorname{Bl}_Z(X) \rightarrow X$ along the centre $Z$ is a del Pezzo surface if and only if the residue field of each point is separable and one of the following holds: 
		\begin{enumerate}[(i)]
			\item $Z = \{a\}$ consists of a degree-3 or 6 point. 
			\item $Z = \{a, a^{\prime}\}$ consists of two degree-3 points, with $\kappa(a) \not\simeq \kappa(a^{\prime})$. 
			\item $Z = \{a, a^{\prime}\}$ consists of two degree-3 points, with $\kappa(a) \simeq \kappa(a^{\prime})$ such that $a^{\prime} \notin C_a$ and $a \notin C_{a^\prime}$.
		\end{enumerate}
	\end{theorem}
	
	\begin{proof}
		Denote by $u: X_{k^{\mathrm{sep}}} \rightarrow X$ the base change to the separable closure, where $X_{k^{\mathrm{sep}}} \simeq \mathbb{P}^2_{k^{\mathrm{sep}}}$. Since blowing-up commutes with flat base change, we are tasked with finding conditions under which the points in the preimage $u^{-1}(Z)$ lie in general position, in the sense of Theorem~\ref{P^2_del_Pezzo}. 
		
		In Proposition~\ref{Smooth_Blowing_Up}, we have seen that all points $a_i \in Z$ must have separable residue fields, since otherwise the resulting blowing-up is non-smooth. The degree of each point is a multiple of 3, and the total degree is at most 8, by Proposition~\ref{Non-Split-Chow} and Proposition~\ref{degree_bounded}, respectively. This leaves the possibilities of blowing up at $Z = \{a\}$ consisting of a single point of degree 3 or 6, or at two points $Z = \{a,a^{\prime}\}$ of degree 3 each. By Proposition~\ref{6pts_general_position}(i), the points in the preimage of $Z = \{a\}$ always lie in general position. Consequently, the blowing-up is always a del Pezzo surface. 
		
		When blowing up at the centre $ Z = \{a, a^{\prime}\}$ with points of degree 3, we have shown that the six points in the fibre are in general position if and only if they do not lie in configuration \textsc{H}, as seen in the proof of Proposition~\ref{6pts_general_position}. If $\kappa(a) \not\simeq \kappa(a^{\prime})$, this is always the case. Finally, if the residue fields are isomorphic, then by Proposition~\ref{nodal_curve}(ii), the necessary and sufficient condition is $a^{\prime} \notin C_{a}$ and, by symmetry, $a\notin C_{a^{\prime}}$.
	\end{proof}


\end{document}